\documentclass[a4paper,10pt]{amsart}
\usepackage{graphicx,verbatim}
\usepackage{amssymb}
\usepackage{amsthm}
\usepackage{amsmath}
\usepackage{epstopdf}
\usepackage{textgreek}
\usepackage{amsfonts}
\usepackage{mathrsfs}
\usepackage{lscape} 
\usepackage[vcentermath,enableskew]{youngtab}
\usepackage{tikz,tikz-cd}
\usetikzlibrary{arrows,positioning,automata,shadows,fit,shapes}
\usepackage[english]{babel}

\newtheorem{theorem}{Theorem}[section]
\newtheorem{definition}[theorem]{Definition}
\newtheorem{example}[theorem]{Example}
\newtheorem{corollary}[theorem]{Corollary}
\newtheorem{lemma}[theorem]{Lemma}
\newtheorem{proposition}[theorem]{Proposition}
\newtheorem{remark}[theorem]{Remark}
\newtheorem*{theorem*}{Theorem}
\newtheorem*{lemma*}{Lemma}
\newtheorem{theoremalpha}{Theorem}

\DeclareMathOperator{\irr}{irr}
\DeclareMathOperator{\End}{End}
\DeclareMathOperator{\im}{im}
\DeclareMathOperator{\tr}{tr}
\DeclareMathOperator{\str}{str}
\DeclareMathOperator{\sgn}{sgn}
\DeclareMathOperator{\Res}{Res}
\DeclareMathOperator{\cont}{cont}
\DeclareMathOperator{\Ind}{Ind}
\DeclareMathOperator{\St}{St}
\DeclareMathOperator{\Sspec}{Sspec}
\DeclareMathOperator{\Scont}{Scont}

\DeclareMathOperator{\odd}{odd}
\DeclareMathOperator{\even}{even}
\DeclareMathOperator{\Hom}{Hom}
\DeclareMathOperator{\Span}{span}

\title[Realising the projective representation of $S_n$]{Dirac cohomology, the projective supermodules of the symmetric group and the Vogan morphism}

\date{\today}
\author{Kieran Calvert}

\begin{document}
\maketitle
\begin{abstract}In this paper we will derive an explicit description of the genuine projective representations of the symmetric group $S_n$ using Dirac cohomology and the branching graph for the irreducible genuine projective representations of $S_n$. In \cite{CH15} Ciubotaru and He, using the extended Dirac index, showed that the characters of the projective representations of $S_n$ are related to the characters of elliptic graded modules. 
We derived the branching graph using Dirac theory and combinatorics relating to the cohomology of Borel varieties $\mathcal{B}_e$ of $\mathfrak{g}$ and were able to use Dirac cohomology to construct an explicit model for the projective representations. We also described Vogan's morphism for Hecke algebras in type A using spectrum data of the Jucys-Murphy elements. \end{abstract} 

\begin{section}{Introduction}

The characters of the projective representations of the symmetric group were initially described by Schur (cf. \cite{S89}). Nazarov \cite{N90} produced an orthogonal form for the irreducible projective representations of the symmetric group, which descended from a study of the projective representations of the hyperoctahedral group \cite{N97}. Okounkov and Vershik \cite{OV96} developed a new approach to studying the representations of the symmetric group via Jucys-Murphy elements. This approach was later applied to the projective representations of $S_n$ by Vershik and Sergeev \cite{VS08} although there appears to be a flaw with their calculation of the spectrum data. The book written by Humphreys and Hoffman \cite{HH92} gives a clear and comprehensive source for the basics of projective representations of the symmetric group. We will describe the action of the Jucys-Murphy elements using Dirac cohomology, providing an alternative proof for the spectrum calculation. We will then describe the genuine projective representations, giving the action of $\tilde{S}_n$ by matrices.

Dirac operators for spinors on the Riemannian symmetric space $G/K$ originated with Atiyah and Schmid \cite{AS77} and Parthasarathy \cite{P72}. Huang and Pandzic \cite{HP02} continued the study of Dirac cohomology for $(\mathfrak{g},K)$-modules of real reductive groups. Barbasch, Ciubotaru and Trapa \cite{BCT12} then developed a p-adic analog: Dirac cohomology for graded Hecke algebras. This was extended to symplectic reflection algebras \cite{C16}. Ciubotaru and He \cite{CH15}  introduces extended Dirac cohomology which relates the tempered modules of the graded Hecke algebra with the irreducible representations of $\tilde{W}$. A general umbrella framework for all of these examples was described by Flake \cite{F16}. 

Notably in \cite{CH15} the functor taking modules to their extended Dirac index is exact. Combining this with combinatorial results from Garcia and Procesi \cite{GP92} on the restriction of the cohomology groups of the Lie algebra $\mathfrak{g}$ associated to $S_n$ we will be able to deduce the branching graph for the irreducible genuine projective representation of $S_n$. This is described in Section \ref{branching}.

\begin{theoremalpha} The branching rules of the genuine projective irreducible $\tilde{S}_n$-supermodules are:

$$Res^{n}_{n-1}  \tilde{\tau}_\lambda = 
\begin{cases} 
2\tilde{\tau}_{\lambda^{(1)}}\oplus 2\tilde{\tau}_{\lambda^{(2)}}\oplus \ldots \oplus  2\tilde{\tau}_{\lambda^{(r-1)}}\oplus  \tilde{\tau}_{\lambda^{(r)}} & \text{ if } n-r \text{ is odd and } \lambda_r = 1, \\
2\tilde{\tau}_{\lambda^{(1)}}\oplus 2\tilde{\tau}_{\lambda^{(2)}}\oplus \ldots \oplus  2\tilde{\tau}_{\lambda^{(r-1)}}\oplus  2\tilde{\tau}_{\lambda^{(r)}} & \text{ if } n-r \text{ is odd and } \lambda_r > 1, \\
\tilde{\tau}_{\lambda^{(1)}}\oplus \tilde{\tau}_{\lambda^{(2)}}\oplus \ldots \oplus  \tilde{\tau}_{\lambda^{(r-1)}}\oplus  \tilde{\tau}_{\lambda^{(r)}} & \text{ if } n-r \text{ is even.}\end{cases}$$
Here $\lambda$ is always a strict partition.
 \end{theoremalpha} 
Depending on a reduction rule, defined in \cite{GP92}, which gives partitions $\lambda^{(i)}$ from $\lambda$.
We prove that the Casimir element  $\Omega_{\tilde{S}_n}$  introduced in \cite{BCT12} for $\tilde{S}_n$ acts by the second power polynomial in the Jucys-Murphy elements.
\begin{lemma*} On genuine projective representations $\Omega_{\tilde{S}_n} =\frac{1}{2} \sum_{i=1}^n M_k^2$. Here  $M_k$ denotes the Jucys-Murphy elements for $\tilde{S}_n$. \end{lemma*}
In \cite[3.5]{BCT12} the authors give another description of how the Casimir element $\Omega_{\tilde{S}_n}$ acts on Dirac cohomology. 
In section \ref{spectrum}, we will combine this lemma, our branching graph result and a result from \cite{BCT12} to prove inductively that the spectrum data for $\tilde{S}_n$ on projective representations is equivalent to content data on shifted Young diagrams. Therefore we give a different proof of the result claimed in \cite{VS08}, utilising Dirac cohomology and combinatorics from Garcia and Procesi \cite{GP92}. 
Content data on shifted Young diagrams is easy to compute and so this provides a concise way to describe how the Jucys-Murphy elements in $\tilde{S}_n$ act on a representation $V_\lambda$ corresponding to the shifted partition $\lambda$.

 \begin{theoremalpha} The set $\Sspec(n) \subset \mathbb{N}^n$, descending from eigenvalues of Jucys-Murphy elements, is equal to the set $\Scont(n)$, a combinatorial construction from contents of shifted Young tableaux. Further,
$$\Sspec(n,\lambda) \cong \Scont(n,\lambda).$$ 
\end{theoremalpha} 
We also prove a formula stated in \cite{CH15} without proof, involving the dimension of $\tilde{\tau}_\lambda$. In section \ref{explicit} we will describe how one gets from the spectrum data to an explicit description of the genuine projective supermodules of $\tilde{S}_n$ and the action of $\tilde{S}_n$ in matrix form.
 Then finally in Section \ref{vogan} we use the results from Section \ref{spectrum} on the action of the Jucys-Murphy elements to give an explicit description the Vogan morphism for graded Hecke algebras in type A introduced in \cite{BCT12}.
 
 \begin{theoremalpha} Vogan's map projected onto the group algebra $\mathcal{T}_n \cong \tilde{S}_n/ \langle z +1 \rangle$, $\underline{\zeta}: Z(\mathbb{H}) \to Z[\mathcal{T}_n]$, is surjective onto the even part of the centre. Also $\underline{\zeta}$ kills all odd polynomials in $Z(\mathbb{H})$.\end{theoremalpha}

\begin{subsection}*{Acknowledgements}

I would like to personally thank Dan Ciubotaru for the guidance in this project and his continued encouragement and patience. \end{subsection}

\begin{subsection}*{Grants}
During the preparation of this paper the author was supported by EPSRC grant EP/M508111/1 and held a Mark Sadler scholarship at Balliol College, Oxford. 
\end{subsection}
\end{section}
\begin{section}{Definitions} 

We fix a real root system $\{V, \Phi, V^\vee, \Phi^\vee\}$, where $\Phi$ is the set of roots inside the real vector space $
V$, we have a perfect bilinear pairing $B( \ , \ ) : V \times V^\vee \to \mathbb{R}$ The roots are in bijections with the coroots such that $B( \alpha , \alpha^\vee ) = 2$. We define reflections $$s_\alpha : V \to V, \hspace{0.75cm} s_\alpha (v) = v - B( v, \alpha^\vee  )\alpha.$$ Let $W$ be the subgroup of $GL(V)$ generated by the reflections $s_\alpha$ for all $\alpha \in \Phi$.
Fix a $W$-invariant bilinear form on $V$, $( \ , \ )$ on $V$, a set of positive (resp. simple) roots $\Phi^+$ (resp. $\Pi$). The reflections $s_\alpha$ for $\alpha \in \Pi$ generate $W$.


\begin{definition} We let $\delta$ denote the automorphism $W$ given by the conjugation by $w_0$ on $V$, where $w_0$ is the longest element in $W$. Then let $W_\#$ be the extension of $W$ by $\delta$. Explicitly $$W_\# = \langle W , \delta : \delta^2 = 1, \delta w \delta = \delta(w) \text{ for all } w \in W \rangle.$$
\end{definition} 

\begin{example} Let $W$ be the symmetric group $S_n$. 
The underlying root space is $V = \Span_{\mathbb{C}} \{\xi_1,\ldots ,\xi_n\}$ and we will fix the simple roots $\Pi = \{\alpha_i =\xi_i-\xi_{i+1} : i = 1,..,n-1\},$ and positive roots $\Phi^+ = \{\xi_i-\xi_j : i <j \}$. The bilinear form is $( \xi_i , \xi_j ) = \delta_{ij}$. We have the presentation;
 $$S_n = \langle s_1,\ldots ,s_{n-1} : (s_is_j)^{m(i,j)} = 1 \rangle.$$
Here  $m(i,i) = 1$ and $m(i,j) = -( \alpha_i , \alpha_j )( \alpha_j , \alpha_i ) +2$ if $i\neq j$ \cite{M74}. The automorphism $\delta$ come from the action of $-w_0$ on $\Phi$ hence $\delta(\alpha_i) = \alpha_{n-i}$, and $\delta(s_\alpha) = s_{\delta(\alpha)}$.\end{example}

\begin{definition} The graded Hecke algebra $\mathbb{H}$ associated to a Weyl group $W$, $V$ and $( \ ,\ )$ is the associative unital $\mathbb{C}$-algebra generated by $\{ \xi \in V\}$ and $\{w\in W\}$ such that there exist injections, induced by the inclusion of generators,
$$S(V) \hookrightarrow \mathbb{H},$$
$$\mathbb{C}[W] \hookrightarrow \mathbb{H}.$$
Furthermore, for every simple root $\alpha \in \Pi$, there exist cross-relations,
$$\xi \cdot s_\alpha - s_\alpha \cdot s_\alpha(\xi) = ( \alpha , \xi ).$$ 
\end{definition} 
Note that since we are interested in $S_n$ we do not introduce parameters for $\mathbb{H}$.

We define $\mathbb{H}_\#$ to be an extension  of the graded Hecke algebra by $\delta$. Explicitly,
$$ \mathbb{H}_\# = \langle \mathbb{H} , \delta \rangle = \langle \mathbb{H}, \delta : \delta^2 = 1, \delta h \delta = \delta(h) \text{ for all } h \in \mathbb{H} \rangle.  $$  
As a vector space $\mathbb{H}_\#$ is isomorphic to $\mathbb{H} \oplus \mathbb{H}$.

\begin{definition}\label{Clifford} The Clifford algebra  $C(V)$ defined by $V$ and the symmetric bilinear form  $( \ , \ )$ on $V$ is the associative unital algebra generated by elements in $V$ such that
$$\xi\xi' + \xi'\xi = 2( \xi' , \xi )  \text{ for all } \xi, \xi' \in V.$$ 
\end{definition}
Note that this differs from some sources, \cite{BCT12} in particular, where the Clifford algebra is defined with a negative form. The theory is identical over $\mathbb{C}$ as one can just multiply the generators by $\sqrt{-1}$. 

\begin{example} For $S_n$, the vector space $V$ has the euclidean norm and is of dimension $n$. In this case we will denote the Clifford algebra for the rank $n$ case by $C(n)$. \end{example}

The Clifford algebra has a natural filtration by $\mathbb{Z}$. Each generator is given degree one and then $C(V)_n$ is the span of all elements of $C(V)$ with degree $n$ or lower. $C(V)$ also has a $\mathbb{Z}/2\mathbb{Z}$ grading,
$$C(V) = C(V)_{\even}\oplus C(V)_{\odd}$$ 
where $C(V)_{\even}$ is the span of all homogeneous elements of even degree and likewise for $C(V)_{\odd}$.  
Define the transpose of $C(V)$ be the antiautomorphism defined by $\xi^t = -\xi$ for all elements of $V$ and let $\epsilon: C(V) \to C(V)$ to the involution which is the identity on $C(V)_{\even}$ and acts as $-1$ on $C(V)_{\odd}$.

\begin{definition} The Pin group is
$$Pin(V) = \{g\in C(V)^\times | g^t = g^{-1} \text{ and } \epsilon(g) \xi g^{-1} \in V, \text{ for all } \xi \in V\}.$$
\end{definition}
This group is a double cover of the orthogonal group on $V$ with projection $p: Pin(V) \to O(V)$,
$$p(g)(\xi) = \epsilon(g)\xi g^{-1}.$$ 
Since the Weyl group is a subgroup of $O(V)$ we can find a double cover of $W$ in the pin group. Let $\tilde{W}$ be the preimage of $W$ via $p$,
$$\tilde{W} = p^{-1} (W).$$

\begin{theorem}\label{morris} \cite[3.2]{M74} The double cover $\tilde{W}$ of  $W$ admits a presentation akin to Coxeter presentations for a Weyl group,

$$\tilde{W} = \langle z , \tilde{s}_\alpha \text{ for all } \alpha \in \Pi : z^2 = 1, (\tilde{s}_\alpha \tilde{s}_\beta) ^{m(\alpha,\beta)} = z^{m(\alpha,\beta)-1}, z \text{ is central } \rangle$$
where $m(\alpha, \alpha) =1$ and $m(\alpha,\beta) = \langle \alpha , \beta \rangle \langle \beta , \alpha \rangle+ 2$. 
\end{theorem} 

\begin{remark}Morris  also showed that one can define a presentation of $W$ with generators for each positive root \cite[3.2]{M74},
$$W = \langle z , \tilde{s}_\alpha \text{ for all } \alpha \in \Phi^+ : \tilde{s}_\alpha^2 = 1, \tilde{s}_\alpha \tilde{s}_\beta \tilde{s}_\beta = z \tilde{s}_\gamma, \gamma = s_\alpha(\beta), z^2 =1, z \text{ is central}\rangle.$$ \end{remark} 
As a subgroup of $O(V)$, we have $W_\# = \langle W ,\delta\rangle$ is equal to $\langle W, -1\rangle$. Hence we can also define $\tilde{W}_\# = p^{-1}(W_\#)$. The non scalar element in the preimage of $-1$ is $\omega$ which can be formulated in terms of an orthonormal basis $\{\xi_1,\ldots ,\xi_n\}$ of $V$, as 
$$\omega = \xi_1\xi_2\ldots \xi_n \in Pin(V).$$

We will also define $\tilde{W}'$ and $\tilde{W}_\#'$ to be
$$\tilde{W}' = \begin{cases} \tilde{W} \cap C(V)_{\even}, & \dim V \text{ even }, \\ \tilde{W}, &  \dim V \text{ odd },  \end{cases}$$

and
$$\tilde{W}_\#' = \begin{cases} \langle \tilde{W} \cap C(V)_{\even}, \omega \rangle, & \dim V even, \\ \tilde{W}_\#,  & \dim V odd. \end{cases}$$

In Section \ref{Dirac} we will utilise the extended Dirac index, introduced by Ciobotaru and He \cite{CH15}. In \cite{CH15} it was crucial that the spinor module considered had a positive and negative part. Since \cite{CH15} consider ungraded modules, in the even dimensional case, the ungraded spinor module does not have a positive and negative part. Hence to manufacture a module with a positive and negative part, one must restrict the  groups $\tilde{W}$ and $\tilde{W}_\#$. This is why we have introduced $\tilde{W}'$ and $\tilde{W}_\#'$. However we will show in Corollary \ref{superdirac} that when one considers supermodules, one no longer needs to restrict $\tilde{W}$. In the super theory case the $C(V)$ module $\mathcal{S}$ always has a positive and negative part. 

We will be interested in the representation theory of $\tilde{S}_n$.
\begin{definition} \cite[p. 303]{S82} Let $\mathcal{T}_n$ be the twisted group algebra of $S_n$ over $\mathbb{C}$. Explicitly $\mathcal{T}_n$ is the unital associative $\mathbb{C}$-algebra generated by $\tau_i$ for $i = 1 \ldots n-1$ such that
$$\tau_i^2 = 1, (\tau_i \tau_j)^{m(i,j)} = (-1)^{m(i,j)-1}.$$
 \end{definition}

\begin{remark} Following the presentation from Theorem \ref{morris} we could label the generators with the simple roots, $\alpha \in \Pi$.

Then $\mathcal{T}_n$ becomes
$$\langle \tau_{\alpha}, \alpha \in \Pi : (\tau_\alpha\tau_\beta)^{m(\alpha,\beta)} = (-1)^{m(\alpha,\beta) -1} \rangle.$$
Finally we also have 
$$\mathcal{T}_n = \langle  \tau_\alpha \text{ for all } \alpha \in \Phi^+ : \tau_\alpha^2 = 1, \tau_\alpha \tau_\beta \tau_\beta = - \tau_\gamma, \gamma = s_\alpha(\beta)\rangle.$$
So by the above presentation, similarly to the symmetric group, one has for each positive root, $\alpha \in \Phi^+$, a \ `pseudo-transposition' $\tau_\alpha$. Note that this $\tau_\alpha$ is equal to the the `transposition' $[ij]$ defined in \cite[3.1]{BK01}.
We will switch between labelings of the generators whenever it clarifies explanation. 
\end{remark}

The genuine projective irreducible representations of $\tilde{S}_n$ are the irreducible representations that do not descend to a representation of $S_n$. The group algebra for $\tilde{S}_n$ decomposes as  
$$\mathbb{C}[\tilde{S}_n] \cong \mathbb{C}[\tilde{S}_n]/(z -1) \oplus \mathbb{C}[\tilde{S}_n ]/(z+1) \cong \mathbb{C}[S_n] \oplus \mathcal{T}_n.$$ If we are interested in the projective representations of $\tilde{S}_n$, i.e. those that do not occurs as representations of $S_n$, then we just need to study the representations of $\mathcal{T}_n$.

Since the super theory for $\tilde{S}_n$ is cleaner than that of the ungraded representation theory we will focus on the supermodules of $\tilde{S}_n$. 
\begin{definition} A vector superspace is a vector space $U$ which is $\mathbb{Z}/2\mathbb{Z}$-graded, $U = U_0 \oplus U_1$. Similarly, a superalgebra is an algebra which is $\mathbb{Z}/2\mathbb{Z}$-graded. \end{definition} 
 For example, the Clifford algebra is a superalgebra, by considering the generators to have degree one. Similarly $\mathcal{T}_n$ is a superalgebra. Again, the generators all have degree one.
 
 \begin{definition}  For a vector superspace $U$ we can define $\End(U)$ to be isomorphic to $\End(U_0 \oplus U_1)$ as an ungraded algebra then as a superalgebra the even elements fix $U_i$ and odd elements take $U_0$ to $U_1$ and $U_1$ to $U_0$. \end{definition}
 \begin{definition} A supermodule of a superalgebra $A$ is a super vector space $U$ along with a graded map from $A \to End(U)$. 
 \end{definition} 
 When one forgets the grading all irreducible supermodules are either still irreducible or decompose into two irreducible parts. It is possible to recover the original nongraded representation theory from the super representation theory if we keep track of whether each irreducible supermodule is reducible as an ungraded module. 
 \begin{definition} Let $U$ be an $A$-supermodule. Then $U$ is of type M if, when one forgets the grading, $U$ is an irreducible $A$ ungraded module. The module $U$ is of type Q if it is reducible as an $A$ ungraded module. \end{definition}   

As an ungraded algebra the Clifford algebra $C(V)$ has, depending on $\dim(V)$, either one or two isomorphism classes of irreducible modules. However when we consider supermodules, $C(V)$ is supersimple; it always has exactly one simple supermodule, $\mathcal{S}$ \cite[12.2.4]{K05}. When $\dim(V)$ is odd then $\mathcal{S}$ is of type Q and when $\dim(V)$ is even the supermodule is of type M.   

The supertrace on the representation $U= U_0 \oplus U_1$ is defined to be zero on odd elements and, on even elements,
$$\str_V(a) = \tr_0 (a) - \tr_1(a)$$
where $tr_i(a)$ is the trace of the matrix corresponding to $a$ restricted to $U_i$. 

We will use three different types of partitions throughout. Let $\lambda = (\lambda_1,..,\lambda_r) \in \mathbb{N}^r$. The set $\lambda$ is a partition of $n$ if $\sum_{i=1}^r \lambda_i = n$ and $\lambda_1\geq \lambda_2 \geq \ldots \geq \lambda_r$. We will denote the set of partitions of $n$ by $\mathcal{P}(n)$. We define the length of a partition, $l(\lambda)$, to be $\#\{\lambda_i : \lambda_i > 0\}$
We say a partition $\lambda \vdash n$ is strict if $\lambda_1>\lambda_2>\ldots >\lambda_r$.
\begin{definition} For a partition $\lambda \vdash n$ we associate to it a set of boxes forming the Young diagram. 
For $\lambda = (\lambda_1,\ldots ,\lambda_r)$, we define the Young diagram to be 
$$\lambda = \{ (i,j) : 1 < i < l(\lambda) \text{ and } 1< j < \lambda_i \}.$$ \end{definition}
Here $(i,j)$ denotes a cell in the $ith$ row and $jth$ column. 
\begin{definition} The set of shifted partitions of $n$ is the same as the set of strict partitions of $n$. However, we will consider a different $Young$ diagram for a shifted partition than the associated Young diagram of a normal partition. We will write $\lambda \vDash n$ to denote a shifted partition of $n$ and $\mathcal{SP}(n)$ is the set of shifted partitions of $n$. \end{definition} 

\begin{definition} Given a shifted partition $\lambda = (\lambda_1,\ldots ,\lambda_r) \vDash n$ we associate a shifted Young diagram,

$$\lambda = \{ (i,j) : 1 < i < l(\lambda) \text{ and } i< j < i+ \lambda_i -1 \}.$$

 \end{definition} 

  A Young tableau of shape $\lambda$ is a numbering of $1$ to $n$ of the Young diagram associated to $\lambda$. We will write $[\lambda]$ for a Young tableau of shape $\lambda$. Similarly 
  $[[\lambda]]$ will denote a numbering of the shifted Young diagram associated to $\lambda \vDash n$. A standard Young tableau (resp. shifted) is a Young tableau (resp. shifted) such that along each row and column the numbers increase. 

 Given a partition $\lambda$ we write $\lambda'$ for its conjugate partition, which is the reflection of the Young diagram in its main diagonal. 
\begin{example} If $\lambda = (3,1)$ then $\lambda' = (2,1,1)$,
$$\lambda = \yng(3,1), \lambda' = \yng(2,1,1).$$ \end{example}

\end{section}

\begin{section}{Extended Dirac cohomology} 

In \cite{CH15} the authors introduced  a notion of extended Dirac cohomology as a variation of Dirac cohomology defined in \cite{BCT12}. An important feature is that the functor from an $\mathbb{H}_\#$-module to its extended Dirac cohomology is exact. Using super theory we show the extended Dirac cohomology could be considered as the super theorem analog of Dirac cohomology.

Let $\rho$ be the diagonal embedding of $\tilde{W}_\#$ into $\mathbb{H}_\#\otimes C(V)$;
$$\rho (w) = p(w) \otimes w, \text{ for all } w \in \tilde{W}_\#.$$
Hence, for an $\mathbb{H}_\#$-module $X$ and a spinor module $\mathcal{S}$, we can consider $X\otimes \mathcal{S}$ as a $\tilde{W}_\#$-module via the map $\rho$.  

Recall $w_0$ is the longest element in $W$. In our case, with $S_n$ and our choice of generators, $w_0$ is the element which takes $(1,2\ldots ,n)$ to $(n,n-1,\ldots ,1)$.
\begin{definition}Let $*$ be the linear anti-automorphism on $\mathbb{H}$ defined by 
$$w^* = w^{-1}, \xi^* = w_0 \delta(\xi)w_0, \text{ for all } w\in W \text{ and } \xi \in V.$$
We can extend $*$ to $\mathbb{H}_\#$ by defining it to fix $\delta$. Given $\xi$, let 
$$\tilde{\xi} = \frac{1}{2}(\xi - \xi^*).$$  \end{definition}

\begin{definition} Let $\{\xi_1,\ldots ,\xi_n\}$ be an orthonormal basis  of $V$. The Dirac element $\mathcal{D}$ in $\mathbb{H}_\#\otimes C(V)$ is defined to be,
$$\mathcal{D} = \sum_{i=1}^{n} \tilde{\xi}_i \otimes \xi_i.$$
\end{definition} 
This is independent of the choice of orthonormal basis. For any $\mathbb{H}_\#$-module $X$ and $C(V)$-module $\mathcal{S}$ the Dirac element $\mathcal{D}$ defines an operator $D: X\otimes \mathcal{S} \to X\otimes \mathcal{S}$. 

As a superalgebra we consider $\mathbb{H}_\#$ to be concentrated in even degree and $C(V)$ to have its usual grading. With this grading $\mathbb{H}_\#\otimes C(V)$ is a superalgebra. Furthermore, since $\mathcal{D}$ is odd, the operator $D$ interchanges the even and odd spaces of $X\otimes \mathcal{S}$,
$$D^\pm : X\otimes \mathcal{S}^\pm \to X\otimes \mathcal{S}^\mp.$$ 
Since as a supermodule we consider $X$ to be concentrated in degree 0, $X\otimes \mathcal{S}^+$ (resp. $X \otimes \mathcal{S}^-$) is the even (resp. odd) space of $X\otimes \mathcal{S}$.

The extended Dirac cohomology of $X$ (with respect to $\mathcal{S}$) is 
$$H^D_\#(X)  = \ker(D_\#) / \ker(D_\#)\cap \im(D_\#).$$
This is a $\tilde{W}_\#$ representation since $\tilde{W}_\#$ commutes with $\mathcal{D}$.
We are also interested in the restriction of $D_\#$ to the even and odd spaces. The subspace $X\otimes \mathcal{S}^\pm$ is always a $\tilde{W}_\#'$-module.  Let,
$$H^{D^\pm}_\#(X) = \ker(D^\pm_\#) / \ker(D^\pm_\#)\cap \im(D^\mp_\#).$$
Note that when $\dim V$ is even $X\otimes \mathcal{S}^\pm$ is not an $\mathbb{H}_\#\otimes C(V)$-module, just an $\mathbb{H}_\#\otimes C(V)_{\even}$-module. The space $H^{D^\pm}_\#(X)$ is a $\tilde{W}_\#'$-module. 

\begin{definition} The extended Dirac index is 
$$I_\#(X) = H^{D^+}_\#(X) - H^{D^-}_\#(X)$$
defined as a virtual $\tilde{W}_\#'$-module. \end{definition}

 \begin{proposition}\label{CH15prop5.10} \cite[5.10]{CH15} Let $X$ be an $\mathbb{H}_\#$-module. The Dirac index can be expressed as a tensor, 
 $I_\#(X) = X \otimes (\mathcal{S}^+ - \mathcal{S}^-).$  Hence it is exact. Furthermore for $w \in \tilde{W}'$,
$$ tr(w, I_\#(X)) = tr(p(w),X)tr(w,S^+ - S^-).$$
 
 \end{proposition}

 We make the further observation that one does not need to split $\tilde{W}_\#$ when $\dim V$ is even if one considers the Grothendieck group of supermodules. This leads to a more elegant formulation of Theorem \ref{CH15prop5.10} using super theory.
 \begin{corollary}\label{superdirac} Let $X$ be an $\mathbb{H}_\#$-supermodule, concentrated in degree zero, and $\mathcal{S}$ be a $C(V)$-supermodule.Then $I_\#(X)$ is a $\tilde{W}_\#$-supermodule and in the Grothendieck group,
 $$I_\#(X) = X \otimes \mathcal{S}.$$
 \end{corollary}

 \begin{proof} From Theorem \ref{CH15prop5.10} we can describe the trace of $w \in \tilde{W}$ on $I_\#(X)$ as a product of the trace on $X$ and the difference of the trace on $\mathcal{S}^+$ and $\mathcal{S}^-$. However this is just the product of the supertrace of the modules $X$ and $\mathcal{S}$;
 $$ \tr(w, I_\#(X)) = \tr(p(w),X)\tr(w,S^+ - S^-) = \str(p(w),X)\str(w,S).$$ \end{proof}
 In light of this corollary, one could consider the extended Dirac index as the Dirac supercohomology. When one considers supermodules, as opposed to ungraded modules, the extended Dirac cohomology is the same as considering the super theory analog of the original Dirac cohomology defined in \cite{BCT12}.

It will be useful to introduce a certain graded module. Let $\mathfrak{g}$ be a complex semisimple Lie algebra with Weyl group $W$.  Then, for an element $e$ in the nilpotent cone $\mathcal{N}$, we let $\mathcal{B}_e$ denote the variety of Borel subalgebras of $\mathfrak{g}$ containing $e$.
Springer \cite{S78} defined a $W$ action on the cohomology groups $H^j(\mathcal{B}_e)$.  The cohomology $H^j(\mathcal{B}_e)$ vanishes unless $j$ is even. 

\begin{definition} Let $e$ be an element in the nilpotent cone $\mathcal{N} \subset \mathfrak{g}$, $\mathcal{B}_e$ the variety of Borel subalgebras containing $e$ and $d_e$ be the dimension of $\mathcal{B}_e$. The sign character of $W$ will be denoted by $\sgn$. Let $\textbf{q}$ be a formal symbol. Springer \cite{S78} constructed an action of $W$ on $H^j(\mathcal{B}_e)$. 
Define the $\textbf{q}$-graded $W$-module.
$$X_{\textbf{q}}(e) = \sum_{i \geq 0} \textbf{q}^{d_e - i} H^{2i}(\mathcal{B}_e) \otimes \sgn.$$ \end{definition}

From here on we will fix $W$ to be $S_n$ and $\mathfrak{g} = \mathfrak{sl}_n$. Because of this we will take results from \cite{CH15} and restrict them to $S_n$ to avoid surplus definitions.  Note that the correspondence from \cite[Appendix A]{CH15}, applies to almost any Weyl group, but outside of type A there is an extra complication involving the component group of the centraliser of $e$. In the full generality of \cite{CH15} the graded modules involved are labelled by a nilpotent element $e$ and an element in this component group.  However for $S_n$ this component group is always trivial.

\begin{lemma} \cite[2.2]{CG97}Two $\textbf{q}$-graded $S_n$-modules $X_{\textbf{q}}(e)$ and $X_{\textbf{q}}(e')$ are isomorphic if and only if $e$ is in the same conjugacy class as $e'$. Hence the isomorphism classes of $X_{\textbf{q}}(e)$ can labelled by nilpotent orbits, these are equivalent to partitions of $n$. 
We write the isomorphism class of $X_{\textbf{q}}(e)$ as $X_{\textbf{q}}(\lambda)$ where $e$ has Jordan form corresponding to a partition $\lambda$. \end{lemma} 
We can specialise $X_{\textbf{q}}(e)$ to a $\mathbb{Z}/2\mathbb{Z}$-graded module as 
$$X_{-1}(e)=  \sum_{i \geq 0} (-1)^{d_e - i} H^{2i}(\mathcal{B}_e) \otimes \sgn.$$

\begin{proposition}\cite[A.3]{CH15}\label{alambda}  For every strict partition $\lambda$ of $n$, let $\mathcal{S}$ be an ungraded spinor module and $X_{-1}(\lambda)$ be the graded module of $S_n$ corresponding to $\lambda$.  Define the $\tilde{S}_n$-module,
$$\tilde{\tau}_\lambda = \frac{1}{a_\lambda} X_{-1}(\lambda) \otimes S,$$ 
where 
$$a_\lambda = \begin{cases} 2^{\frac{l(\lambda)}{2}} & n \text{ even and } l(\lambda) \text{ even}, \\ 2^{\frac{l(\lambda)-1}{2}} & l(\lambda) \text{ odd}, \\ 2^{\frac{l(\lambda)-2}{2}} & n \text{ odd and } l(\lambda) \text{ even}. \end{cases}$$
If $n-l(\lambda)$ is even then $\tilde{\tau}_\lambda$ is an irreducible $\sgn$ self-dual $\tilde{S}_n$-module. If $n-l(\lambda)$ is odd then $\tilde{\tau}_\lambda = \tilde{\tau}^+_\lambda + \tilde{\tau}^-_\lambda$, with $\tilde{\tau}^\pm_\lambda$ irreducible modules which are $\sgn$ dual to each other. \end{proposition}

\begin{corollary} For a strict partition $\lambda$, when one lets $\mathcal{S}$ be a super simple spinor module, $\tilde{\tau}_\lambda$ is always an irreducible $\tilde{S}_n$-supermodule. If $n-l(\lambda)$ is even then $\tilde{\tau}_\lambda$ is type M and if $n-l(\lambda)$ is odd then $\tilde{\tau}_\lambda$ if of type Q. \end{corollary}

\begin{remark}\label{indexkilled} The extended Dirac index of $X_{-1}(\lambda)$, $I_\#(X_{-1}(\lambda))$,  is non-zero if and only if $\lambda$ is strict. This follows from  \cite[6.1]{CH15} and the fact that nilpotent elements that are $\delta$ quasi-distinguished have Jordan form associated to a strict partition.\end{remark}  
For ease of explanation later on in this paper, and motivated by Remark \ref{indexkilled}, we will set $\tilde{\tau}_\lambda = 0$ if $\lambda$ is not a strict partition.

 Finally, when looking at the action of the Jucys-Murphy elements in Section \ref{branching} we will need a lemma about the action of the Casimir elements on the Dirac cohomology. 
 
 \begin{definition} The Casimir element for $\mathbb{H}$ is 
 $$\Omega_\mathbb{H} = \sum_{i=1}^n \xi_i^2,$$
 where $\{\xi_i\}$ is an orthonormal basis for $V$. \end{definition}
 The element $\Omega_{\mathbb{H}}$ is independent of the choice of orthonormal basis, it is central in $\mathbb{H}$ \cite[2.4]{BCT12}. Hence it acts by a scalar for any irreducible representation of $X$. Furthermore, let $(\pi,X)$ be an irreducible $\mathbb{H}$-module. It has an associated central character $\chi_\nu$ which can be defined in terms of $\nu \in \mathbb{H}$. It is shown in \cite[2.5]{BCT12} that $$\pi(\Omega_{\mathbb{H}})= ( \nu , \nu ).$$ 

 \begin{definition}\cite[3.4]{BCT12}The Casimir element for $\tilde{S_n}$ is 
 $$\Omega_{\tilde{S_n}} = \frac{1}{4}\sum_{\begin{substack}{\alpha, \beta \in \Phi^+\\ s_\alpha(\beta) < 0} \end{substack}} |\alpha||\beta| \tau_\alpha \tau_\beta,$$
 where $\tau_\alpha$ is the generator of $\tilde{S_n}$ corresponding to $\alpha$.
 In \cite{BCT12} $\Omega_{\tilde{S}}$ is negative but we define it as positive since we have the positive form on $C(V)$. \end{definition} 
 The Casimir $\Omega_{\tilde{S}_n}$ is central in $\mathbb{C}[\tilde{S}_n]$, \cite{BCT12}. 
  
 \begin{theorem}\label{Dirac}\cite[3.5]{BCT12} As elements of $\mathbb{H}\otimes C(V)$,
  $$\mathcal{D}^2 =  \Omega_\mathbb{H}\otimes 1- \rho(\Omega_{\tilde{S_n}}),$$
 \end{theorem} 
Recall $\rho$ is the diagonal embedding of $\tilde{W}$ into $\mathbb{H}\otimes C(V)$. 
 Note that this statement has a parity difference from the one in \cite{BCT12} but this is due to the sign difference of the Clifford algebra in Definition \ref{Clifford}. 

\begin{definition} \label{Steinberg}\cite{Ci16} We introduce the one dimensional $\mathbb{H}$-module, called the Steinberg module $St$. Here $S_n$ acts on $St$ by the $\sgn$ character and $\xi_i\in S(V)$ act by $\frac{n-1 + 2i}{2}$. %
\end{definition} 
\begin{definition} Let $X_\lambda$ be the parabolically induced $\mathbb{H}$-module from the Steinberg module for $\mathbb{H}_{\lambda_1}\times \ldots \times \mathbb{H}_{\lambda_k}$,
$$X_\lambda = \Ind_{\mathbb{H}_{\lambda_1}\times \ldots \times \mathbb{H}_{\lambda_k}}^\mathbb{H} (St_{\lambda_1} \otimes \ldots \otimes St_{\lambda_k}).$$ \end{definition}

In \cite[5.8]{BCT12} it is shown that, inside the Dirac cohomology of $X_\lambda$, one can find the genuine projective irreducible representation $V_\lambda$ of isomorphism class associated to a shifted partition $\lambda \vDash n$.
Since $H_D(X)$ is a quotient of $\ker(D)$, $D^2$ acts by zero on $H_D(X)$.

\begin{corollary} On the Dirac cohomology (resp. extended Dirac index), or any submodule found inside it, namely $V_\lambda$ (resp. $\tilde{\tau}_\lambda$), 
$$\Omega_\mathbb{H}\otimes 1=  \rho(\Omega_{\tilde{W}}).$$
\end{corollary}
\begin{proof} This follows from $D^2$ acting by zero and the equation given in Theorem \ref{Dirac}.\end{proof}

\end{section}
\begin{section}{Branching graph for $\tilde{S}_n$}\label{branching}

In section \ref{spectrum} we will need the branching graph of $\tilde{S}_n$. Mainly, we will need to know modules that occur in the restriction of the irreducible representations $\tilde{\tau}_\lambda$. However, in this section we provide arguments for the whole branching graph of the genuine projective representations of $\tilde{S}_n$, this is the branching graph of $\mathcal{T}_n$. We will derive this branching graph from Theorem \cite[A.4]{CH15} and a branching result from Garcia and Procesi \cite{GP92} on certain graded $S_n$-modules.

We know we can find $\tilde{\tau}_\lambda$ in $\mathcal{S} \otimes X_{-1}(\lambda)$. Tensoring with spinor modules is exact and since the restriction rules of spinor modules are straightforward, all that is left to understand is the restriction of $X_{-1}(\lambda)$.

Garcia and Procesi \cite{GP92} studied a very similar graded module. It is the module $X_{\textbf{q}}(\lambda)$ but with the reverse $\textbf{q}$ grading. 
\begin{definition} \cite[I.7]{GP92} Let  $p^\lambda (q)$ be the character of the graded $S_n$-module 
$$\sum_{i\geq0} \textbf{q}^i H^{2i}(\mathcal{B}_\lambda ).$$
\end{definition} 

In \cite{GP92} a reduction rule for partitions is defined. Given a partition $\lambda \vdash n$ this rule outputs a set, potentially with multiples, of partitions $\lambda^{(i)} \vdash n-1$. 

\begin{definition} \cite[1.1]{GP92} Let $\lambda = (\lambda_1,\ldots ,\lambda_r)$ be a fixed partition of $n$ and $\lambda' = (\lambda_1',\ldots ,\lambda_r')$ be its conjugate. 

For $1 \leq i \leq l(\lambda)$ define the integer $a_i$ by the condition
$$\lambda_{a_i}' \geq i > \lambda_{a_i+1}'.$$
Given this integer, let $\lambda^{(i)}$ be the partition created when one removes a block from the bottom of column $a_i$ of the Young diagram associated to $\lambda$. \end{definition}
Note that here the inequality differs from \cite{GP92} but this is due to the fact that Garcia and Procesi define their partitions to increase.

Given $\lambda \vdash n$, let $\Theta^n_{n-1}(\lambda)$ be the set of partitions of $n-1$ which one can create from the Young diagram of $\lambda$ by removing a single box.

\begin{lemma}\label{strictset} For a strict partition $\lambda = (\lambda_1,\ldots ,\lambda_r) \vdash n$
we have $$\lambda^{(i)} = (\lambda_1,\ldots ,\lambda_{i-1},\lambda_{i} -1 , \lambda_{i+1},\ldots ,\lambda_r).$$ 
Hence $\{\lambda^{(i)} : i =1,\ldots ,l(\lambda) \}$ is equal to the set $\Theta^n_{n-1}(\lambda)$. \end{lemma} 

\begin{proof} 

Because $\lambda$ is strict, the column of $\lambda$ which is first larger than $i$, $\lambda_{a_i}'$, will always be the column that has a block from $\lambda_1,\ldots ,\lambda_i$ and no others. So removing a block from row $\lambda_i$ is equivalent to removing a block from column $\lambda_{a_i}'$. Hence $\lambda^{(i)} = (\lambda_1,\ldots ,\lambda_{i-1},\lambda_{i} -1 , \lambda_{i+1},\ldots ,\lambda_r)$. The second statement follows since we have an explicit description of $\{\lambda^{(i)} : i =1,\ldots ,l(\lambda) \}$. \end{proof} 

Given a module or character $A$ of either $S_n$, $\tilde{S}_n$ or $C(n)$,  let $\Res_{n-1}^n(A)$ denote the restriction to the rank $({n-1})$ object, that is the restriction to $S_{n-1}$, $\tilde{S}_{n-1}$ or $C(n-1)$.

\begin{theorem}\label{GPbranching} \cite[3.3]{GP92} The restriction of the $S_n$ $\textbf{q}$-graded character $p^\lambda(\textbf{q})$ to $S_{n-1}$ is a sum of $\textbf{q}^{i-1}p^{\lambda^{(i)}}(\textbf{q})$;
$$\Res^n_{n-1}p^\lambda(\textbf{q}) = \sum_{i=1}^{l(\lambda)} \textbf{q}^{i-1} p^{\lambda^{(i)}}(\textbf{q}).$$
\end{theorem} 

\begin{remark}\label{reverse} If we let $\chi_{\textbf{q}} (\lambda)$ be the character of $X_{\textbf{q}}(\lambda)$ then
$$\chi_{\textbf{q}} (\lambda ) = p^\lambda (\boldsymbol{\frac{1}{q}}) q^{n(\lambda)} $$ 
where $n(\lambda) = \sum_{i=1}^{l(\lambda)} (i-1)\lambda_i$. \end{remark}

We can combine Remark \ref{reverse} with Theorem \ref{GPbranching} to understand the branching rules for $X_{\textbf{q}}(\lambda)$.

\begin{lemma}\label{qres} Let $\lambda$ be a strict partition. The restriction of the $S_n$-module $X_{\textbf{q}}(\lambda)$ to $S_{n-1}$ is
$$Res^{n}_{n-1} X_{\textbf{q}}(\lambda)  =  \bigoplus_{i=1}^{l(\lambda)} \textbf{q}^{2i-2}X_{\textbf{q}}(\lambda^{(i)}).$$
\end{lemma} 

\begin{proof} By Remark \ref{reverse} one can write $\chi_\textbf{q}(\lambda) = p^\lambda (\boldsymbol{\frac{1}{q}}) q^{n(\lambda)}$. Then, we can use the restriction rules given by Theorem \ref{GPbranching} to restrict $\chi_\textbf{q}(\lambda)$ in terms of the characters $p^{\lambda^{(i)}}(\textbf{q})$. Remark \ref{reverse} can be used again to rewrite this in terms of characters $\chi_{\textbf{q}}(\lambda^{(i)})$. The coefficient one gets is  $\textbf{q}^{i-1+ n(\lambda) - n(\lambda^{(i)})}$. Writing out the definition of $n(\lambda)$; 
$$n(\lambda) - n(\lambda^{(i)}) = \sum_{j=1}^l(\lambda)(j-1)\lambda_j - \sum_{j=1}^{l(\lambda^{(i)})} (j-1)\lambda^{(i)}_j.$$ However, by Lemma \ref{strictset}, $\lambda$ and $\lambda^{(i)}$ differ by only one entry for $\lambda$ strict. This difference is in the $ith$ entry. Therefore $n(\lambda)$ and $n(\lambda^{(i)})$ differ by $i-1$. Hence $i-1+ n(\lambda) - n(\lambda^{(i)}) = 2i-2$.  \end{proof}


\begin{lemma}\label{summands} Let $\lambda$ be a strict partition. The restriction of the graded $S_n$-module $X_{-1}(\lambda)$ is,
$$Res^{n}_{n-1} X_{-1}(\lambda) = \bigoplus_{i=1}^{l(\lambda)} X_{-1}(\lambda^{(i)}).$$ \end{lemma} 
\begin{proof} This is the specialisation of Lemma \ref{qres} to $\textbf{q} = -1$. \end{proof}
Recall the definition of $\tilde{\tau}_\lambda$ \cite[A.3]{BCT12}; $$\tilde{\tau}_\lambda = \frac{1}{a_\lambda} X_{-1}(\lambda) \otimes \mathcal{S}.$$ The set $\{\tilde{\tau}_\lambda : \lambda \text{ is strict}\}$ is a transversal for the irreducible genuine projective supermodules of $S_n$.

Remark \ref{indexkilled} states that tensoring with the spinor kills $X_{-1}(\mu)$ if $\mu$ is not strict. So we can describe which modules will occur in the restriction of $\tilde{\tau}_\lambda$. Since $a_\lambda\tilde{\tau}_\lambda =\mathcal{S}\otimes X_{-1}(\lambda)$

\begin{lemma} Let $\lambda$ be a strict partition. For a fixed $i$, if $\lambda^{(i)}$ is strict then $\tilde{\tau}_\lambda^{(i)}$ is a summand of $\Res_{n-1}^n(\tilde{\tau}_\lambda$). \end{lemma}

\begin{proof} We note that $a_\lambda \tilde{\tau}_\lambda = X_{-1}(\lambda) \otimes \mathcal{S}$. Here, $\mathcal{S}$ (resp. $\mathcal{S}'$) is the spinor supermodule for the Clifford algebra  $C(n)$ (resp. $C(n-1)$). Hence restricting $a_\lambda \tilde{\tau_\lambda} $ to $\tilde{S}_{n-1}$ is equivalent to restricting $X_{-1}(\lambda)$ to $S_{n-1}$ and $\mathcal{S}$ to $C(n-1)$. Therefore

$$a_\lambda Res^{n}_{n-1} \tilde{\tau}_\lambda = Res^{n}_{n-1}\mathcal{S} \otimes Res^{n}_{n-1} X_{-1}(\lambda) $$
$$= c \mathcal{S}' \otimes \bigoplus_{i=1}^{l(\lambda)} X_{-1}(\lambda^{(i)}).$$
$$= \bigoplus_{{\lambda^{(i)}} \text{ which are strict }} c\mathcal{S}' \otimes X_{-1}(\lambda^{(i)}).$$
Here, $c \in \mathbb{Z}\setminus\{0\}$ and will be determined in Theorem \ref{branchingrules}.
Hence, for every strict $\lambda^{(i)}$ the $\tilde{S}_n$ supermodule $\tilde{\tau}_{\lambda^{(i)}}$ occurs as a summand.\end{proof} 
With the information we have it is possible to describe the explicit branching graph for the supermodules of $\mathcal{T}_n$ and hence the branching graph for genuine projective supermodules of $\tilde{S}_n$.

Recall that if $\lambda$ is not a strict partition we defined $\tilde{\tau}_\lambda = 0$ .
\begin{theorem}\label{branchingrules}The branching rules of the genuine projective irreducible $\tilde{S}_n$-supermodules are

$$Res^{n}_{n-1}  \tilde{\tau}_\lambda = 
\begin{cases} 
2\tilde{\tau}_{\lambda^{(1)}}\oplus 2\tilde{\tau}_{\lambda^{(2)}}\oplus \ldots \oplus  2\tilde{\tau}_{\lambda^{(r-1)}}\oplus  \tilde{\tau}_{\lambda^{(r)}} & \text{ if } n-r \text{ is odd and } \lambda_r = 1, \\
2\tilde{\tau}_{\lambda^{(1)}}\oplus 2\tilde{\tau}_{\lambda^{(2)}}\oplus \ldots \oplus  2\tilde{\tau}_{\lambda^{(r-1)}}\oplus  2\tilde{\tau}_{\lambda^{(r)}} & \text{ if } n-r \text{ is odd and } \lambda_r > 1, \\
\tilde{\tau}_{\lambda^{(1)}}\oplus \tilde{\tau}_{\lambda^{(2)}}\oplus \ldots \oplus  \tilde{\tau}_{\lambda^{(r-1)}}\oplus  \tilde{\tau}_{\lambda^{(r)}} & \text{ if } n-r \text{ is even.}\end{cases}$$
Here $\lambda$ is always a strict partition.
 \end{theorem} 

\begin{proof} By Lemma \ref{summands} we already know which summands will occur. Hence calculating the multiplicities is all that is required. We know the multiplicities of the restriction of the spinor $S_n$; $2$ if $n$ is even, $1$ otherwise. Also, by Lemma \ref{summands} the multiplicities in the restriction of $X_{-1}(\lambda)$ are always $1$. Hence, we can find the multiplicities for  $\tilde{\tau}_\lambda$ by comparing $\frac{1}{a_\lambda}$ and $\frac{1}{a_{\lambda^{(i)}}}$ from Proposition \ref{alambda}. This is a simple calculation on the eight cases
$$\frac{1}{a_\lambda} = \begin{cases}
\frac{1}{2} \frac{1}{a_{\lambda^{(i)}}} & n \text{ even, } l(\lambda) \text{ even and } l(\lambda^{(i)}) = l(\lambda),\\
\frac{1}{2}\frac{1}{a_{\lambda^{(i)}}}  & n \text{ even, } l(\lambda) \text{ even and } l(\lambda^{(i)}) = l(\lambda)-1,\\
\frac{1}{a_{\lambda^{(i)}}} & n \text{ even, } l(\lambda) \text{ odd and } l(\lambda^{(i)}) = l(\lambda),\\
\frac{1}{2}\frac{1}{a_{\lambda^{(i)}}} & n \text{ even, } l(\lambda) \text{ odd and } l(\lambda^{(i)}) = l(\lambda)-1,\\
2\frac{1}{a_{\lambda^{(i)}}}& n \text{ odd, } l(\lambda) \text{ even and } l(\lambda^{(i)}) = l(\lambda),\\
\frac{1}{a_{\lambda^{(i)}}} & n \text{ odd, } l(\lambda) \text{ even and } l(\lambda^{(i)}) = l(\lambda)-1,\\
\frac{1}{a_{\lambda^{(i)}}} & n \text{ odd, } l(\lambda) \text{ odd and } l(\lambda^{(i)}) = l(\lambda),\\
\frac{1}{a_{\lambda^{(i)}}} & n \text{ odd, } l(\lambda) \text{ odd and } l(\lambda^{(i)}) = l(\lambda)-1.\\ \end{cases}$$
Note that $l(\lambda^{(i)}) = l(\lambda)-1$ occurs once if and only if $\lambda_r = 1$, in which case it is the partition $\lambda^{(r)}$ which has shorter length.
\end{proof}

\end{section}

\begin{section}{Spectrum data for $\tilde{S}_n$}\label{spectrum}

 In Section \ref{explicit} we will be able to construct the genuine projective representations. However, the raw information that we need to do this is the action of the Jucys-Murphy elements squared; this is what we call the spectrum data. We will prove that this is equivalent to a function on the contents of the Young tableaux for $\lambda$.

As described before $\mathcal{T}_n$ is generated as an associative algebra by $\tau_\alpha$ for $\alpha \in \Phi^+$.

 \begin{definition}\cite[3.1]{BK01} The Jucys-Murphy elements in $\mathcal{T}_n$ for $i=1,\ldots ,n$ are, $$M_i = \sum_{\alpha \in \Phi^+ : ( \alpha , \xi_i ) < 0} \tau_\alpha.$$ This is the same construction as the Jucys-Murphy elements for $S_n$. We have just replaced transpositions with pseudo-transpositions.  \end{definition} 
Note that Brundan and Kleschev define the Jucys-Murphy elements differently. They use the generators $\tau_i$ and replace $\tau_\alpha$ with $[ij]$ when $\alpha = \xi_i-\xi_j$ and then define $M_i =  \sum_{j <i} [ij]$. These Jucys-Murphy elements are the same just with different labelling of generators.  
 
\begin{remark}\label{anti-commute} The Jucys-Murphys elements anti-commute \cite[3.1]{BK01}, that is 
 $$M_kM_l = -M_lM_k \text{ if } k \neq l.$$ \end{remark}
 
 \begin{lemma} \cite[3.2]{BK01} The even centre $Z(\mathcal{T}_n)_0$ of $\mathcal{T}_n$ is spanned by the set of symmetric polynomials of the Jucys-Murphy elements. \end{lemma}

 Schur (cf. \cite{S89}) defined all of the genuine projective characters for $\tilde{S}_n$, and he showed that these correspond to the set of shifted partitions $\lambda \in \mathcal{SP}(n)$.
 
\begin{definition}  Let $V_\lambda$ denote any genuine projective $\tilde{S}_n$-module which has character corresponding to $\lambda \in \mathcal{SP}(n)$. \end{definition}

It will be useful, for notation, to introduce a function $q:\mathbb{Z} \to \mathbb{Z}$ such that $q(m) = m(m+1)$.

\begin{definition}\label{spec}\cite{VS08} 
We say $\alpha = (\alpha_1,\ldots ,\alpha_n)$ is in $\Sspec(n)$ if there exists a vector $v$ in an irreducible genuine projective representation $V$ of $\tilde{S}_n$ such that: 
$$M_k^2 v = \frac{1}{2}q(a_k) v \text{ for all } k=1,\ldots ,n.$$
$\Sspec(n, \lambda)$ is the restriction of $\Sspec(n)$ by considering vectors in any $V_\lambda$, $\lambda \in \mathcal{SP}(n)$.  \end{definition} 

\begin{definition} Let $\lambda \in \mathcal{SP}(n)$ be a shifted partition.The content of a box $(i,j)$ contained in a shifted Young diagram $\lambda$ is its distance from the main diagonal.
$$\cont(i,j) = j-i$$ \end{definition}

\begin{definition} We introduce the set of shifted content vectors $\Scont(n)$. 
A vector $\beta = (\beta_1,\ldots ,\beta_n)$ is associated to a standard shifted tableaux $[[\lambda]]$ if, for all $i=1,\ldots ,n$, $\beta_i$ is equal to the content of the box labelled i in $[[\lambda]]$. 
$\Scont(n)$ is the set of vectors which are associated to a standard shifted tableau of size n. Similarly $\Scont(n,\lambda)$, is the set of vectors associated to standard shifted tableaux of shape $\lambda$. \end{definition}
In this section, our goal will be to prove that 
$$\Sspec(n) =\Scont(n)$$
 and $\Sspec(n,\lambda) = \Scont(n,\lambda)$.

Recall the Casimir elements for $\mathbb{H}$ and $\tilde{S}_n$, $\Omega_{\mathbb{H}}$ and $\Omega_{\tilde{S}_n}$ respectively. 
Our technique for describing the set $\Sspec(n)$ will be largely based on using two different descriptions of the action of the Casimir element $\Omega_{\tilde{S}_n}$. The first one, descending from Dirac cohomology, states that $\Omega_{\tilde{S}_n} = \Omega_{\mathbb{H}}$ on $\tilde{\tau}_\lambda$. The second will be linked to the Jucys-Murphys elements. Using these descriptions we will be able to show $\Sspec(n) = \Scont(n)$ inductively. 

Recall the definition of $X_\lambda$;
$$X_\lambda = \Ind_{\mathbb{H}_{\lambda_1}\times..,\times\mathbb{H}_{\lambda_k}}^\mathbb{H}(\St_{\lambda_1} \boxtimes \ldots  \boxtimes \St_{\lambda_k}).$$ 

A central character is a map $Z(\mathbb{H}) = S(V)^W \to \mathbb{C}$. The standard representations $X_\lambda$ have central characters $\chi_\lambda$.  The vector space $\Hom(S(V)^W, \mathbb{C})$ can be associated with $V^*/W$ via evaluating polynomials in $S(V)/W$ at an element in $V^*/W$. Hence a central character, $\chi$, corresponds to an element $\nu \in V^*/W$. 

\begin{lemma} Let $St_m$ be the Steinberg module for $\mathbb{H}_m$. Let $\nu_{St_m}$ be the element in $S(V)^*/W$ corresponding to $\chi_{St_m}$ . Let $e_1,\ldots ,e_m \in V^*$ be a a dual basis of the basis  $x_1,..,x_m$ of $V$. 
The element $\nu_{\St_m}$ is,
$$\nu_{\St_m} = -\frac{m-1}{2}e_1 + -\frac{m-3}{2}e_2+\ldots +\frac{m-3}{2}e_{m-1} + \frac{m-1}{2}e_m = \sum_{i=1}^m \frac{2i-m}{2}e_i.$$
\end{lemma} 

 \begin{corollary} The element $\nu_\lambda$ defining the central character of $X_\lambda$ is 
$$\nu = \sum_{i=1}^{l(\lambda)} \nu_{\St_{\lambda_i}}.$$ \end{corollary}

 \begin{lemma}\label{HOmega} \cite[2.5]{BCT12} On the representation  $X_\lambda$ of the graded Hecke algebra $\mathbb{H}$,
 $$\Omega_\mathbb{H} = ( \nu , \nu ).$$
 \end{lemma} 

\begin{lemma}\label{content} On the representation $\tilde{\tau}_\lambda$ corresponding to $\lambda = (\lambda_1,\ldots ,\lambda_r)$, 
$$\Omega_\mathbb{H}  = \sum_{i=1}^{l(\lambda)} \sum_{j=1}^{\lambda_i} (\frac{2j - \lambda_i}{2})^2.$$

Furthermore, this can be reformulated in terms of the content of $\lambda$,
$$ \Omega_\mathbb{H} =\frac{1}{4} \sum_{(i,j) \in \lambda}  q(\cont(i,j)).$$\end{lemma}

\begin{proof} The first part can be proved by noticing that, when one forgets the grading, the module $X_{-1}(\lambda)$ is just the parabolically induced module $X_\lambda$, see \cite{BM83}. Hence, one can use Lemma \ref{HOmega} to show how $\Omega_\mathbb{H}$ acts. The constant $( \nu, \nu )$ is exactly the sum we  stated.
The second part follows by a simple induction. One can show $\sum_{j=1}^{\lambda_i} (2j - \lambda_i)^2= \sum_{j=1}^{\lambda_i} j(j+1) = \sum_{j=1}^{\lambda_k} q(\cont(k,j))$, splitting the cases when $\lambda_i$ is odd or even. This will be covered in more detail in Section \ref{vogan}. \end{proof} 

\begin{lemma}\label{squares} In the algebra $\mathcal{T}_n$, the Casimir element, $\Omega_{\tilde{S}_n} =\frac{1}{2} \sum_{i=1}^n M_k^2$. Here  $M_k$ denotes the Jucys-Murphy elements for $\mathcal{T}_n$. \end{lemma} 
\begin{proof} 

$$\Omega_{\tilde{S}_n}= \frac{1}{4}\sum_{\substack{\alpha,\beta\in\Phi^+ \\ s_\alpha(\beta)<0}} |\alpha||\beta|\tau_\alpha \tau_\beta $$
 $$=\frac{1}{4}\sum_{\alpha,\beta\in \Phi^+} |\alpha||\beta|\tau_\alpha \tau_\beta= \frac{1}{4}(\sum_{\alpha\in\Phi^+}|\alpha|\tau_\alpha)^2$$
 $$=\frac{1}{4} (\sum_{k=1}^{n}|\alpha|M_k)^2= \frac{1}{2} \sum_{i=1}^n M_k^2$$
 
In the second equality, we use $\tau_\alpha \tau_\beta = - \tau_\beta \tau_{s_\alpha(\beta)}$ for $\alpha$ and $\beta$ such that $s_\alpha(\beta)>0$.
The third equality uses the fact that every $\sum_{i=1}^n M_k = \sum_{\alpha \in \Phi^+} \tau_\alpha$ and the last equality uses the fact that the Jucys-Murphy elements anti-commute in $\mathcal{T}_n$ \cite[3.1]{BK01}. Hence the square of the sum is equal to the sum of the squares. 
\end{proof}

\begin{theorem}\label{spectrumdata} The set $\Sspec(n) \subset \mathbb{N}^n$ is equal to the set $\Scont(n)$. Further,
$$\Sspec(n,\lambda) = \Scont(n,\lambda).$$ 

\end{theorem}

\begin{proof}

We prove the theorem by induction, the base case being trivial. Suppose for every shifted partition $\mu \in \mathcal{SP}(n-1)$, $\Scont(n-1, \mu) = \Sspec(n-1, \mu)$. 
Let us fix a shifted partition $\lambda$ of $n$, and consider $\tilde{\tau}_\lambda$ the irreducible representation of $\mathcal{T}_n$ associated to $\lambda$.

Let $\alpha = (\alpha_1,\ldots ,\alpha_n)\in \Sspec(n , \lambda)$. By definition there exists a vector $v_\alpha \in \tilde{\tau}_\lambda$ such that $M_j^2(v_\alpha) =q(\alpha_j )v_\alpha \text{ for all } $j$.$

The restriction of $\tilde{\tau}_\lambda$ to $\mathcal{T}_{n-1}$ is the direct sum, $\bigoplus_{i =1}^{l(\lambda)} \tilde{\tau}_{\lambda^{(i)}}$. Also $v_\alpha$ is an eigenvector for the Jucys-Murphys elements. Hence $v_\alpha \in \tilde{\tau}_{\lambda^{(j)}} \subset \tilde{\tau}_\lambda$ for some fixed $j$. We know, by the inductive hypothesis, $(\alpha_1,\ldots ,\alpha_{n-1}) \in \Sspec(n-1,\lambda^{(j)}) = \Scont(n-1,\lambda^{(j)})$. Explicitly, this means there exists a standard numbering $[[\lambda^{(j)}]]$ of $\lambda^{(j)}$ which corresponds to $(\alpha_1,\ldots ,\alpha_{n-1})$. 

Now using Lemma \ref{content}, Lemma \ref{squares} and Theorem \ref{Dirac}, we can describe the action of the Casimir elements of $\tilde{S}_n$ and $\tilde{S}_{n-1}$ on $v_\alpha$. Since $v_\alpha$ is in $\tilde{\tau}_\lambda$ ,
 $$ \frac{1}{2}\sum_{i=1}^n M_k^2 v_\alpha = \Omega_{\tilde{S}_n} v_\alpha= \Omega_{\mathbb{H}_n}v_\alpha =\frac{1}{4} \sum_{(i,j) \in \lambda}  \cont(i,j)(\cont(i,j)+1)v_\alpha. $$ 
 However $v_\alpha$ is also contained in $\tilde{\tau}_{\lambda^{(j)}}\subset \tilde{\tau}_\lambda$ so,

 $$\frac{1}{2} \sum_{i=1}^{n-1} M_k^2v_\alpha = \Omega_{\tilde{S}_{n-1}} v_\alpha= \Omega_{\mathbb{H}_{n-1}}v_\alpha =\frac{1}{4} \sum_{(i,j) \in \lambda^{(j)}}  \cont(i,j)(\cont(i,j)+1)v_\alpha.$$

  the difference of these two Casimir elements is $\frac{1}{2}M_n^2$. Hence
 $$M_n^2 v_\alpha = \left(\frac{1}{2}\sum_{(i,j) \in \lambda}  \cont(i,j)(\cont(i,j)+1) - \frac{1}{2}\sum_{(i,j) \in \lambda^{(j)}}  \cont(i,j)(\cont(i,j)+1) \right) v_\alpha$$
 $$ =\frac{1}{2} \cont(k,l)(\cont(k,l)+ 1)v_\alpha,$$  where $(k,l)$ is the block of $ \lambda $ not included in $ \lambda^{(i)}.$
Taking the standard shifted tableau $[[\lambda^{(j)}]]_{(\alpha_1,\ldots ,\alpha_{n-1})}$ associated to $(\alpha_1,\ldots ,\alpha_{n-1})$ and adding the block $(k,l)$ labelled with the number $n$ creates a standard shifted tableau $[[\lambda]]_\alpha$ of shape $\lambda$. By the above argument this standard shifted tableau has content vector equal to $\alpha \in \Sspec(n,\lambda)$. Hence $\alpha \in \Scont(n,\lambda)$.

For the reverse direction, we use an almost identical argument. Suppose $\beta =(\beta_1,\ldots ,\beta_n) \in \Scont(n,\lambda)$  corresponds to a standard shifted tableau $[[\lambda]]_\beta$. By restriction and inductive hypothesis, there exists a vector $v_{(\beta_1,\ldots ,\beta_{n-1})} \in \tilde{\tau}_{\lambda^{(j)}} \subset \tilde{\tau}_\lambda$ such that $M_k^2 v_{(\beta_1,\ldots ,\beta_{n-1})} = q(\beta_k) v_{(\beta_1,\ldots ,\beta_{n-1})}$. 
Again using the fact that $M_n^2$ acts by the difference of the Casimir elements $\Omega_{S_n}$ and $\Omega_{S_{n-1}}$, we get, $$M_n^2 v_{(\beta_1,\ldots ,\beta_{n-1})} = 2(\Omega_{\tilde{S}_n} - \Omega_{\tilde{S}_{n-1}})   = \frac{1}{2}q(\cont(n))v_{(\beta_1,\ldots ,\beta_{n-1})},$$ where $\cont(n)$ is the content of the box labelled by $n$ in $[[\lambda]]_\beta$. Hence $\beta \in Spec(n,\lambda)$. 

\end{proof}

\end{section}
\begin{section}{Explicit representation from spectrum data}\label{explicit}

This is, in essence, the same as the last part of \cite{VS08}. We describe how to explicitly construct the genuine irreducible representations from the spectrum data we calculated in the previous section.

Let  $\langle M_1,\ldots ,M_n \rangle$ be the subalgebra of $\mathbb{C}[\tilde{S}_n]$ generated by the Jucys-Murphy elements $M_i$. 

\begin{lemma} Let $V$ be a genuine irreducible projective representation of $\tilde{S}_n$. Then the restriction of $V$ to  $\langle M_1,\ldots ,M_n \rangle$ is a direct sum of the common eigenspaces of $M_i^2$.
\end{lemma}
\begin{proof} The $M_i^2$ all commute so we can separate $V$ into its common eigenspaces. Now $M_i$ fix these spaces so the common eigenspaces are  $\langle M_1,\ldots ,M_n \rangle$ submodules of $V$.\end{proof}
\begin{lemma} 
Let $U$ be an  $\langle M_1,\ldots ,M_n \rangle$-supermodule which occurs as a common eigenspace of $M_i^2$ in some $V$. Then the the representation $U$ of  $\langle M_1,\ldots ,M_n \rangle$ factors through the Clifford algebra of rank $r$, where $r = |\{M_i^2 \text{ with non zero eigenvalue}\}| = n - l(\lambda)$. \end{lemma}

\begin{proof} The set of Jucys-Murphy elements anti-commute and on $U$, by Definition \ref{spec}, $M_i^2 = q(\alpha_i)$ for some $\alpha = (\alpha_1,\ldots ,\alpha_n) \in \Sspec(n)$. By sending $M_i$ to $\frac{c_i}{\sqrt{q(\alpha_i)}}$ for $\alpha_i \neq 0$, where $c_i$ are the Clifford generators for the Clifford algebra  the action of  $\langle M_1,\ldots ,M_n \rangle$ on $U$ factors through the Clifford algebra of rank $\# \{\alpha_i \neq 0\}$. \end{proof}

\begin{corollary}\label{JMsummand} Let $\tilde{\tau}_\lambda$ be the genuine irreducible supermodule of $\tilde{S}_n$, associated to $\lambda$. Then the common eigenspaces for $M_i^2$ are labelled by the spectrum data $\alpha \in \Sspec(n,\lambda)$. Hence, when considered as a $\langle M_1,\ldots ,M_n \rangle$-supermodule,
$$\tilde{\tau}_\lambda = \bigoplus_{\alpha \in \Sspec(n,\lambda)} U_\alpha.$$\end{corollary}

If we take any genuine projective irreducible representation $V$ of $\tilde{S}_n$ then when we restrict it to the subalgebra $\langle M_1,\ldots ,M_n \rangle$, generated by the Jucys-Murphy elements, $V$ decomposes as the eigenspaces of $M_i^2$. Furthermore these eigenspaces are irreducible $\langle M_1,\ldots ,M_n \rangle$-modules which can be considered as spinor modules for the Clifford algebra of rank $n- l(\lambda)$.

We will build a genuine irreducible representation $V_\lambda$ of $\tilde{S}_n$ which will be isomorphic to $\tilde{\tau}_\lambda$. Since the set $\Sspec(n)$ is a subset of $\mathbb{N}^n$ then $S_n$ naturally acts on this set, further, $S_n$ preserves $\Sspec(n,\lambda)$. We understand how $M_i$ act on an irreducible $\mathcal{T}_n$-module corresponding to $\lambda$. Motivated by studying how $\tau_i$ acts on $\tilde{\tau}_\lambda$, we will add in the action of $\tau_i$ on $V_\lambda$.


 \begin{lemma}\label{negative} Using the decomposition in Corollary \ref{JMsummand}, the subspace $(\tau_i - \frac{M_i -M_{i+1}}{q(a_i) - q(a_{i+1})} ) U_\alpha$ of $\tilde{\tau}_\lambda$ is fixed by $\langle M_j \rangle$ and
$$\left(\tau_i - \frac{M_i -M_{i+1}}{q(a_i) - q(a_{i+1})} \right) U_\alpha =U_{s_i(\alpha)}.$$ \end{lemma} 
\begin{proof} One can see that for every $u \in U_\alpha$, 
$$M_j \left(\tau_i - \frac{M_i -M_{i+1}}{q(a_i) - q(a_{i+1})} \right)u = - \left(\tau_i - \frac{M_i -M_{i+1}}{q(a_i) - q(a_{i+1})} \right) M_{s_i(j)}u.$$ 

Therefore we can conclude that $\left(\tau_i - \frac{M_i -M_{i+1}}{q(a_i) - q(a_{i+1})} \right) U_\alpha$ is the subspace we labelled $U_{s_i(\alpha)}$. 
Note that if $s_i(\alpha) \notin \Sspec(n,\lambda)$ then $\left(\tau_i - \frac{M_i -M_{i+1}}{q(a_i) - q(a_{i+1})} \right) U_\alpha = 0$ hence, in this case, $\tau_i = \frac{M_i -M_{i+1}}{q(a_i) - q(a_{i+1})}$ on $U_\alpha$.
\end{proof}

\begin{definition} If $\alpha$ and $s_i(\alpha)$ are in $\Sspec(n,\lambda)$ then define $P_{s_i}:V_\alpha \to V_{s_i(\alpha)}$ to be the vector space isomorphism such that 
$$M_{s_i(j)}^2 (P_{s_i} (v)) = P_{s_i} M_j v \text{ for all } v \in V_\alpha.$$ 
 \end{definition} 
 Note that $P_{s_i}$ is not an $\langle M_1,..,M_n\rangle$ isomorphism since it interchanges the action of different Jucys-Murphy elements.

Now we can describe $V_\lambda$ as a $\tilde{S}_n$-module.
\begin{definition}\label{actionoftau} 
Let $V_\lambda =    \sum_{\alpha \in \Sspec(n, \lambda)} V_\alpha$ as supermodules of $\langle M_1,\ldots ,M_n\rangle \subset \tilde{S}_n$.  Let $\alpha = (\alpha_1,\ldots ,\alpha_n) \in \Sspec(n, \lambda)$. The action of $\tau_i$, on $V_\alpha$ is defined in the following way. If $s_i(\alpha) \in \Sspec(n)$ then,

 $$\tau_i = \frac{M_i - M_{i+1}}{q(a_i) -q(a_{i+1})}- \left(1 - \frac{q(a_i) +q(a_{i+1})}{(q(a_i) - q(a_{i+1}))^2}\right)^{-\frac{1}{2}} P_{s_i} .$$
 
If $s_i(\alpha ) \notin \Sspec(n, \lambda)$ then $(1 - \frac{q(a_i) +q(a_{i+1})}{(q(a_i) - q(a_{i+1})^2}))= 0$
 and on $V_\alpha$,
 $$\tau_i = \frac{M_i - M_{i+1}}{q(a_i) -q(a_{i+1})}.$$
 \end{definition}
 
\begin{theorem}  For the set of shifted tableaux $\mathcal{SP}(n)$, $\{V_\lambda: \lambda \in \mathcal{SP}(n)\}$ is a full set of representatives of the irreducible genuine projective supermodules of $\tilde{S}_n$. 
 \end{theorem}
 
 \begin{proof} All we need to show is that the supermodules $V_\lambda$ are isomorphic to $\tilde{\tau}_\lambda$. Corollary \ref{JMsummand} shows that they are isomorphic as $\langle M_1,\ldots,M_n\rangle$-supermodules. The following arguments will show that the action of $\tau_i$ agree on both $V_\lambda$ and $\tilde{\tau}_\lambda$.
 If $s_i(\alpha) \notin \Sspec(n,\lambda)$ then Lemma \ref{negative} shows that $\tau_i$ acts on both supermodules identically. In the case that $(\tau_i - \frac{M_i -M_{i+1}}{q(a_i) - q(a_{i+1})} )$ interchanges $V_\alpha$ and $S_{s_i(\alpha)}$ one observes, that applying this operator twice gives 
 $$\left(\tau_i - \frac{M_i -M_{i+1}}{q(a_i) - q(a_{i+1})}\right) ^2 = \left(1 - \frac{q(a_i) +q(a_{i+1})}{(q(a_i) - q(a_{i+1})^2})\right)Id$$
 on $V_\alpha$.  Let $p = 1 - \frac{q(a_i) +q(a_{i+1})}{(q(a_i) - q(a_{i+1}))^2} $.
  Furthermore, again by Lemma \ref{negative}
  $$\tau_i - \frac{M_i -M_{i+1}}{q(a_i) - q(a_{i+1})}  = -\frac{1}{\sqrt{p}} P_{s_i}. $$ 
Rearranging this equation gives the result,
 $$\tau_i = \frac{M_i - M_{i+1}}{q(a_i) -q(a_{i+1})}- \left(1 - \frac{q(a_i) +q(a_{i+1})}{(q(a_i) - q(a_{i+1}))^2}\right)^{-\frac{1}{2}} P_{s_i} .$$
 So again the action of $\tau_i$ is identical on both modules.
 \end{proof}
 
 \begin{lemma} \cite[1.2]{S80} The hook length formula for shifted tableaux gives an explicit description for the number of standard shifted tableaux for a shifted partition $\lambda = (\lambda_1,\ldots ,\lambda_r\}  \vDash n$,
 $$\# \{\text{ standard shifted tableaux of shape }\lambda\} = \frac{n!}{\lambda_1!\ldots \lambda_r!} \prod_{i=1}^{r} \frac{\lambda_i - \lambda_j}{\lambda_i +\lambda_j}.$$ \end{lemma} 
 
 \begin{corollary}\label{dimension} The dimension of $V_\lambda$ is 
 $$2^{\lceil \frac{n-l(\lambda)}{2}\rceil}  \frac{n!}{\lambda_1!\ldots \lambda_r!} \Pi_{i=1}^{r} \frac{\lambda_i - \lambda_j}{\lambda_i +\lambda_j}.$$ \end{corollary}

 \begin{proof} We built $V_\lambda$ as a direct sum of simple supermodules $\mathcal{S}_{n-l(\lambda)}$, which have dimension $2^{\lceil \frac{n-l(\lambda)}{2}\rceil}$. The number of $\mathcal{S}_{n-l(\lambda)}$ in the sum is equal to the number of standard shifted tableaux, which is $\frac{n!}{\lambda_1!\ldots \lambda_r!} \prod_{i=1}^{r} \frac{\lambda_i - \lambda_j}{\lambda_i +\lambda_j}$. \end{proof}

Corollary \ref{dimension} shows that the ungraded representations $\tilde{\sigma}_\lambda$ in \cite{CH15} have dimension 
$$ 2^{\lfloor \frac{n-l(\lambda)}{2}\rfloor}  \frac{n!}{\lambda_1!\ldots \lambda_r!} \prod_{i=1}^{r} \frac{\lambda_i - \lambda_j}{\lambda_i +\lambda_j}.$$
This is stated but not proved in \cite[Example 6.9]{CH15}.

We will illustrate, with an example, how one can create an explicit model of a representation corresponding to a shifted partition $\lambda$.
\begin{example} 
 Take $\tilde{S}_4$ and the shifted tableau $\lambda = (3,1) = \young(\hfil \hfil \hfil ,:\hfil ).$ 
 There are two standard shifted tableaux, namely,
  $$\young(123,:4) \hspace{5mm} \young(124,:3) .$$
 with content vectors 
 $$\alpha_1 = (0,1,2,0), \alpha_2= (0,1,0,2).$$
 
 In this case $n - l(\lambda) = 2$ so our building blocks are $C(2)$-supermodules. Let $\mathcal{S}$ be the (1,1) dimensional supermodule of $C(2) = \langle c_1 , c_2 : c_1c_2 = -c_2c_1 , c_i^2 =1\rangle$ where $$c_1 =  \begin{bmatrix} 0 & 1 \\ 1 &0\end{bmatrix} \text{ and } c_2 =  \begin{bmatrix} 0 & \sqrt{-1} \\ -\sqrt{-1} &0\end{bmatrix}.$$ 
 As a superspace  $V_{(3,1)} =\mathcal{S}^1 \oplus \mathcal{S}^2$, where $\mathcal{S}^j \cong \mathcal{S}$ as a $C(2)$-module. For the action of the Jucys-Murphy elements, $M_1= 0$, $M_2 = (\sqrt{2}c_1,\sqrt{2} c_1)$, $M_3 = (\sqrt{6}c_2,0)$ and $M_4 = (0,\sqrt{6}c_2)$. 
 Finally, to calculate the action of $\tau_i$, since $s_i(\alpha_j) \in \Sspec(4, (3,1)) $ for $i = 1,2,$ then $\tau_1$ and $ \tau_2$ fix $\mathcal{S}^1$ and $\mathcal{S}^2$. Explicitly, using Definition \ref{actionoftau}:
 $$\tau_1 =  \begin{bmatrix} 0 & \sqrt{2} & 0 & 0 \\ \sqrt{2} & 0 & 0 & 0 \\ 0 & 0 & 0 & \sqrt{2} \\ 0 & 0 & \sqrt{2} & 0 \end{bmatrix},$$
 $$\tau_2 = \begin{bmatrix} 0 & \frac{1}{2}(\sqrt{-6} - \sqrt{2}) & 0 & 0 \\ -\frac{1}{2}(\sqrt{-6}+\sqrt{2}) & 0 & 0 & 0 \\ 0 & 0 & 0 & \sqrt{2} \\ 0 & 0 & \sqrt{2} & 0 \end{bmatrix}.$$ 
 Now $\tau_3$ interchanges $\alpha_1$ and $\alpha_2$, hence again using Theorem \ref{actionoftau} 
 $$\tau_3 = \begin{bmatrix} 0 & \frac{\sqrt{-6}}{3} & \sqrt{\frac{2}{3}} & 0 \\ - \frac{\sqrt{-6}}{3} & 0 & 0 & \sqrt{\frac{2}{3}}  \\ \sqrt{\frac{2}{3}}  & 0 & 0 & \frac{\sqrt{-6}}{3} \\ 0 & \sqrt{\frac{2}{3}}  & - \frac{\sqrt{-6}}{3} & 0 \end{bmatrix}.$$
 
This gives an explicit realisation of of the supermodule $V_{(3,1)}$ on the four dimensional space spanned by $\{v_1,v_2,v_3,v_4\}$, by the above matrices. Although note that the even space is $V_{(3,1)}^0 = span\{v_1,v_4\}$ and the odd space is $V_{(3,1)}^1 = span\{v_2,v_3\}$. 
 
 \end{example}

\end{section}

\begin{section}{Description of Vogan's morphism}\label{vogan}
This section will be dedicated to describing Vogan's morphism in type A. We will use the description of the action of $Z(\mathcal{T}_n)_0$ on $V_\lambda$ from Section \ref{spectrum} and the central character $\chi_\lambda$ of $X_\lambda$ describing the actions of $Z(\mathbb{H})$ on $X_\lambda$.

 Vogan's morphism for graded Hecke algebras is the algebra homomorphism $\zeta: Z(\mathbb{H}) \to Z(\mathbb{C}[\tilde{S}_n])$ defined such that $\zeta(z)$ for $z \in Z(\mathbb{H})$ is the unique element in $Z(\tilde{S}_n)$ such that $$z\otimes 1 = \rho (\zeta(z)) + Da + aD$$ as elements of $\mathbb{H} \otimes C(V)$, for some $a \in \mathbb{H}\otimes C(V)$. This morphism was introduced in \cite{BCT12} and is inspired by Vogan's morphism for real reductive groups. It occurs naturally in Dirac theory for graded Hecke algebras but as far as the author is aware it has not been described unlike in Dirac theory for real reductive groups.
 We will describe 
$$\underline{\zeta}: Z(\mathbb{H}) \to Z(\mathcal{T}_n).$$
This is $\zeta$ composed with the projection from $\mathbb{C}[\tilde{S}_n]$ to $\mathcal{T}_n$.

The following lemma is well known.
\begin{lemma}\label{density} Let $A$ be a semisimple finite dimensional $\mathbb{C}$-algebra and let $a$ and $b \in Z(A)$.
Suppose for all irreducible representations $\sigma$ of $A$, $$\sigma(a)=\sigma(b),$$
then $a = b$.
\end{lemma} 

Since $\mathcal{T}_n$ is a semisimple finite dimensional algebra over $\mathbb{C}$, we will use the Lemma \ref{density} and our description of how the centres act on $V_\lambda$ to describe $\zeta$.

 Let $\lambda= (\lambda_1,\ldots ,\lambda_k)$ be a partition. Recall the central character $\chi_\lambda$ of the $\mathbb{H}$-module $X_\lambda$  is a function $\chi_\lambda: Z(\mathbb{H}) \to \mathbb{C}$. Since $Z(\mathbb{H}) = S(V)^{S_n}$ any central character can be defined by a $S_n$ orbit of $V^*$. Let $\{y_i\}$ be a dual basis of $\{x_i\}$ then the element defining $\chi_\lambda$ is 
 $$\nu_\lambda = \frac{1}{2} \sum_{i=1}^k \sum_{j =1}^{\lambda_i} (-\lambda_i +2j-1)y_{\lambda_{i+1}+j}.$$
 
 \begin{example} Take $\lambda = (5,3,2)$, then the coefficients of $\nu$ can be encoded in the following labelled diagram similar to a Young diagram for $(5,3,2)$, but translated to be vertically symmetrical.

 \centering
  \begin{tikzpicture}

    \node[] (1) at (-2,0) {$-2$};
    \node[] (2) at (-1,0) {$-1$};
    \node[] (3) at (0,0) {$0$};
     \node[] (4) at (1,0) {$1$};
      \node[] (5) at (2,0) {$2$};

    \node[] (1a) at (-1,-0.5) {$-1$};
    \node[] (2a) at (0,-.5) {$0$};
    \node[] (3a) at (1,-.5) {$1$};
    
    \node[] (1c) at (-0.5,-1) {$-\frac{1}{2}$};
    \node[] (2c) at (0.5,-1) {$\frac{1}{2}$};

\end{tikzpicture}
\end{example}

\begin{definition} The dual map $\underline{\zeta}^*$ takes the irreducible representations of $\mathcal{T}_n$ to central characters of $\mathbb{H}$. 
Let $\sigma \in \irr(\mathcal{T}_n)$ and let $p \in Z(\mathbb{H})$. 

The central character $\underline{\zeta}^* (\sigma): Z(\mathbb{H}) \to \mathbb{C}$ is
$$ \underline{\zeta}^*(\sigma)(f) = \sigma(\underline{\zeta}(p)).$$
\end{definition}

We can describe $\underline{\zeta}^*$ explicitly.
\begin{lemma} Let $(V_\lambda,\sigma_\lambda)$ be the$\mathcal{T}_n$ representation described in Section \ref{explicit} corresponding to a strict partition $\lambda$, let $\chi_\lambda$ be the central character of $X_\lambda$.
Then as described $$\underline{\zeta}^* (V_\lambda) = \chi_\lambda.$$
\end{lemma}
This is a corollary of \cite[4.4]{BCT12}.

\begin{lemma}\label{equivalent}  Let $P_{\mathbb{H}} \in Z(\mathbb{H})$ and $P_{\mathcal{T}_n} \in Z(\mathcal{T}_n)$. 
If, for all $\lambda$ strict,
$$V_\lambda(P_{\mathcal{T}_n}) = \chi_\lambda(P_\mathbb{H})$$
then $\underline{\zeta}(P_\mathbb{H}) = P_{\mathcal{T}_n}.$
\end{lemma}

\begin{proof}

The set of $V_\lambda$ for shifted partitions is a full set of irreducible representations of $\mathcal{T}_n$. Also $\underline{\zeta}^*(V_\lambda) = \chi_\lambda$. So the assumption is equivalent to, for all $\sigma \in \irr(\mathcal{T}_n)$, $\sigma(P_{\mathcal{T}_n}) =(\sigma)(\underline{\zeta}(P_{\mathbb{H}})).$ Then applying Lemma \ref{density},
$\underline{\zeta}(P_{\mathbb{H}}) = P_{\mathcal{T}_n}$.

\end{proof}

Now we will define polynomials in $S(V)^{S_n}=Z(\mathbb{H})$ and elements in $Z(\mathcal{T}_n)$ which have the same action. Then we will use Lemma \ref{equivalent} to describe $\underline{\zeta}$.
Recall that the even centre of $\mathcal{T}_n$ is spanned by symmetric polynomials in the squares of the Jucys-Murphys elements $M_i$. 
\begin{definition} Let $M_i$ be the Jucys-Murphy elements for $\mathcal{T}_n$. We define,
$$P_{\mathcal{T}_n}^{2m} = \sum_{i=1}^n (2M_i)^{2m} \in Z(\mathcal{T}_n)_{\even}.$$

The set $\{\xi_1,..,\xi_n\}$ is an orthogonal basis of $V$. Define
$$P_{\mathbb{H}}^{2m} =  \sum_{j=0}^{\lfloor \frac{m}{2} \rfloor} {{m}\choose{2j}} \sum_{i=1}^n (2x_i)^{2m-2j}\in Z(\mathbb{H})_{\even}.$$
\end{definition} 

\begin{theorem}\label{action} For $m, n\in \mathbb{N}$ and for all shifted $\lambda$,
$$\sigma_\lambda(P_{\mathcal{T}_n}^{2m}) = \chi_\lambda(P_{\mathbb{H}}^{2m}).$$
\end{theorem} 

We will delay the proof of this theorem and first state the consequences.
\begin{corollary}We have 
$$\underline{\zeta}(P_{\mathbb{H}}^{2m}) = P_{\mathcal{T}_n}^{2m}.$$ \end{corollary}

\begin{proof}This follows from Lemma \ref{equivalent} and Theorem \ref{action}.\end{proof}
Hence we can describe how $\underline{\zeta}$ acts on $P^{2m}_\mathbb{H}$. Note that this set of polynomials spans the even centre of $\mathbb{H}$ so we have described half of $\underline{\zeta}$.

\begin{lemma} For all $\lambda$ strict and $n,m \in \mathbb{N}$,
$$\chi_\lambda(\sum_{i=1}^n x_i^{2m+1} )= 0.$$
\end{lemma}

\begin{proof} The character $\chi_\lambda$ is defined by $\nu_\lambda$ which for every positive coefficient has a negative coefficient of the same magnitude. Hence any symmetric odd polynomial evaluated on $\nu_\lambda$ is zero. \end{proof}
This gives a description of $\underline{\zeta}$ on the odd part of the centre, $Z(\mathbb{H})_{\odd}$. Hence we have described the action of $\underline{\zeta}$ on $Z(\mathbb{H})$.

\begin{corollary} The map $\underline{\zeta}$ surjects onto the even part of the centre of $\mathcal{T}_n$. Also $\ker(\underline{\zeta}) \supset S(V)_{\odd}$.\end{corollary} 

\begin{proof} The first statement is clear since $Z(\mathcal{T})_{\even}$ is generated by the symmetric polynomials in the squares of the Jucys-Murphy elements. The second statement follows from the fact that every symmetric homogeneous polynomial of odd degree is generated by monomials of power polynomials where there is at least one odd power polynomial. Hence $\underline{\zeta}$, since it is a homomorphism, will kill any homogeneous symmetric polynomial of odd degree. \end{proof}

Note that Kleschev and Brundan \cite[3.2]{BK01} provide a basis for $Z(\mathcal{T}_n)_{\even}$ in the form of products of power polynomials of the $M_i^2$. The set $\{ \Pi_{i=1}^{k} P_{\mathcal{T}_n}^{2m_i} : \sum_{i=1}^{k}2m_i+1 \leq n\}$ is a basis for $Z(\mathcal{T}_n)_{\even}$. Therefore 
$$\{ \underline{\zeta}'(\Pi_{i=1}^{k} P_{\mathbb{H}}^{2m_i}) : \sum_{i=1}^{k}2m_i +1 \leq n\}$$ is a basis of $Z(\mathcal{T}_n)$. 
For a particular $n$ it is possible to find the even elements in the kernel of $\underline{\zeta}'$ but this must be done on a case-by-case basis comparing the action on each $V_\lambda$ and showing it to be zero.

\begin{subsection}{Proof of Theorem \ref{action}}

\begin{definition} Define $\mathcal{P}^1$ to be the set of partitions, of any number, of length 1. $\mathcal{P}^1_n$ is the singleton subset of $\mathcal{P}^1$, consisting of the partition $\lambda = (n).$ \end{definition}

Our first step is to show Theorem \ref{action} can be proved just by considering partitions of length one.
\begin{theorem}\label{action2} Let $\mu \in \mathcal{P}^1$ be a partition of length one then, for all $m \in \mathbb{N}$,
$$\sigma_\mu(P_{\mathcal{T}_n}^{2m}) = \chi_\mu (P_{\mathbb{H}}^{2m}).$$
\end{theorem}

\begin{lemma} Theorem \ref{action2} implies Theorem \ref{action}, that is, it is enough to show the result on modules corresponding to the partition of length one. \end{lemma} 

\begin{proof} Let $\lambda = (\lambda_1,\lambda_2,\ldots ,\lambda_k)$ be a shifted partition.
We are studying the action of power polynomials. Let $Q_{\mathbb{H}}^j$ be the $j^{th}$ power polynomial in the $x_i$'s, $\sum_{i=1}^n x_i^j$. Let $\chi_{\lambda_1 \times \ldots \times \lambda_k}$ be the central character of the $\mathbb{H}_{\lambda_1} \times \ldots \times \mathbb{H}_{\lambda_k}$-module $St_{\lambda_1} \otimes \ldots \otimes St_{\lambda_k}$ and $\chi_{\lambda_l}$ be the central character of the $\mathbb{H}_{\lambda_l}$-module $St_{\lambda_l}$. We consider $\mathbb{H}_{\lambda_l}$ to be embedded in $\mathbb{H}$ such that $x_{\lambda_{l-1}+1}, \ldots, x_{\lambda_l}$ are in the image of $\mathbb{H}_{\lambda_l}$. Since $Q_{\mathbb{H}}^j \in \mathbb{H}_{\lambda_1} \times \ldots \times \mathbb{H}_{\lambda_k}$, then
$$\chi_\lambda(Q_\mathbb{H}^j) = \chi_{\lambda_1 \times \ldots \times \lambda_k}(Q_\mathbb{H}^j) =  \sum_{l=1}^k \chi_{\lambda_l} (Q_{\mathbb{H}_{\lambda_l}}^j).$$
Similarly, for $Q_{\mathcal{T}_n}^j = \sum_{i=1}^n M_i^j$,
$$\sigma_\lambda (Q_{\mathcal{T}_n}^j) = \sum_{l=1}^k \sigma_{\lambda_l}(Q_{\mathcal{T}_{\lambda_l}}^j).$$

Therefore if we can prove the result of the theorem for every $\mu = (\lambda_l) \in \mathcal{P}^1$, then using the above decomposition of both $\chi_\lambda(Q_\mathbb{H}^j)$ and $\sigma_\lambda (Q_{\mathcal{T}_n}^j)$ we can extend this to $\sigma_{\lambda}$ and $\chi_\lambda$ for every strict $\lambda$. \end{proof}

\begin{lemma}\label{cont} Let $\lambda = (\lambda_1,..,\lambda_n) \Vdash n$ be a shifted partition, with associated shifted Young diagram $\lambda = \{(i,j) : i = 1,..,l(\lambda) \text{ and } j = i -1 ,\ldots ,i -1+ \lambda_i. \}$. Then 
$$\sigma_\lambda(P_{\mathcal{T}_n}^{2m}) =\sigma_\lambda(\sum_{i=1}^n (2M_i)^{2m}) = \sum_{i=1}^n (cont(i))^m(cont(i)+1)^m.$$
\end{lemma} 

\begin{proof} This is a restatement of Theorem \ref{spectrumdata} where we described the action of $M_i^2$ on $V_\lambda$. Recall that the result showed that on a certain subspace $M_i^2$ acted by $cont(i)(cont(i)+1)$.  \end{proof}
\begin{corollary}\label{actionT} Fix a $t\in \mathbb{Z}$ and let $\mu = (\mu_i) = (t) \Vdash t$. Then 
$$\sigma_\mu(\sum_{i=1}^t (2M_i)^{2m}) = \sum_{i=1}^t (i-1)^m(i)^m.$$ \end{corollary}
\begin{proof} This follows from Lemma \ref{cont} since for $\mu$, $\cont(i) = i-1$. \end{proof}

\begin{lemma}\label{actionH} Fix a $t \in \mathbb{Z}$ Let $\mu = (\mu_1)= (t) = \Vdash t$. Then
$$\chi_\mu (\sum_{i=1}^t  (2x_i)^{2m} ) = \sum_{i=1}^{t} (2(i-1)-(t-1))^{2m}.$$
\end{lemma} 

\begin{proof} The central character $\chi_\mu$ is defined by $\nu_\mu = \sum_{i=1}^{t} (2(i-1) - (t-1)y_i$.  The result follows by evaluating $\sum_{i=1}^t  (2x_i)^{2m}$ at $\nu_\mu$.   \end{proof} 

\begin{proof}[Proof of Theorem \ref{action2}] Fix $m \in \mathbb{N}$.
We prove this statement by induction on $n$ with steps of length two. Note that for $n=0$ and $n=1$, all operators act by zero, hence the base cases are trivial. 
Suppose the result is true for $n-2$, so: $$\sigma_{(n-2)}(P_{\mathcal{T}_{n-2}}^{2m}) = \chi_{(n-2)} (P_{\mathbb{H}}^{2m}).$$
Now, considering that $\sigma_{(n-2)}$ is the restriction of $\sigma_{(n)}$ to $\mathcal{T}_{n-2}$,
$$\sigma_{(n)}(P_{\mathcal{T}_{n}}^{2m}) - \sigma_{(n-2)}(P_{\mathcal{T}_{n-2}}^{2m}) = \sigma_{(n)}( (2M_{n-1})^{2m} + (2M_{n})^{2m}).$$
By Corollary \ref{actionT} this is equal to 

\begin{equation*}
\begin{split}
\sum_{i=1}^{t} (i-1)^m(i)^m - \sum_{i=1}^{t-2} (i-1)^m(i)^m & =   (t-2)^m(t-1)^m + (t-1)^m(t)^m\\
& = (t-1)^m ((t-1 + 1)^m + (t-1 -1 )^m ) \\
&= (t-1)^m (\sum_{l=0}^m {m\choose l} (t-1)^{m-l} (1^l + (-1)^l))\\
& = (t-1)^m\sum_{\substack{l \in {0,\ldots ,m} \\\text{ and even} }} 2 {m\choose l} (t-1)^{m-l}\\
 & = \sum_{j = 0}^{\lfloor \frac{m}{2} \rfloor} 2 {m \choose 2j} (t-1)^{2m-2j}.
 \end{split} \end{equation*}

We know that  
$$\chi_{(t)}(P_{\mathbb{H}_t}^{2m}) - \chi_{(t-2)}P_{\mathbb{H}_{t-2}}^{2m},$$
is just the action of the $x_{t}$ and $x_{t-1}$. Expanding out $P_{\mathbb{H}_t}^{2m}$ and explicitly writing $\chi_{(t)}(P_{\mathbb{H}_t}^{2m})$ using Lemma \ref{actionH}, we know that 
 
$$\chi_{(t)}(P_{\mathbb{H}_t}^{2m}) - \chi_{(t-2)}P_{\mathbb{H}_{t-2}}^{2m}= \sum_{j = 0}^{\lfloor \frac{m}{2} \rfloor} 2 {m \choose 2j} (t-1)^{2m-2j}.$$
Therefore we have shown that the inductive step holds. Hence 
$$\sigma_{(n)}(P_{\mathcal{T}_{n-2}}^{2m}) = \chi_{(n)} (P_{\mathbb{H}}^{2m})$$
by induction. \end{proof}

We have proved Theorem \ref{action2} and hence have proved Theorem \ref{action}.

\end{subsection}
\end{section}

\bibliography{bib}{}
\bibliographystyle{abbrv}
\end{document}